\documentclass[review]{elsarticle}

\usepackage{bm,amssymb,amsmath,mathrsfs}
\usepackage{amsthm}
\usepackage{lineno}
\usepackage[usenames]{color}

\usepackage{dcolumn}

\newcommand{\bq}{\begin{equation}}
\newcommand{\eq}{\end{equation}}
\newcommand{\bc}{\begin{center}}
\newcommand{\ec}{\end{center}}
\newcommand{\bit}{\begin{itemize}}
\newcommand{\eit}{\end{itemize}}
\newcommand{\ben}{\begin{enumerate}}
\newcommand{\een}{\end{enumerate}}

\theoremstyle{plain}
\newtheorem{theorem}{Theorem}[section]
\newtheorem*{theorem*}{Theorem}
\newtheorem{proposition}[theorem]{Proposition}

\newtheorem{remark}[theorem]{Remark}
\newtheorem{definition}[theorem]{Definition}
\newtheorem{conjecture}[theorem]{Conjecture}

\begin{document}

\journal{(internal report CC23-8)}

\begin{frontmatter}

\title{Convexity and monotonicity of the probability mass function of the Poisson distribution of order $k$}

\author[cc]{S.~R.~Mane}
\ead{srmane001@gmail.com}
\address[cc]{Convergent Computing Inc., P.~O.~Box 561, Shoreham, NY 11786, USA}

\begin{abstract}
This note focuses on the properties of two blocks of elements of the probability mass function (pmf) of the Poisson distribution of order $k\ge2$.
The first block is the elements for $n\in[1,k]$ and the second block is the elements for $n\in[k+1,2k]$.
It is proved that elements in the first block form an ``absolutely monotonic sequence''
by which is meant that all the finite differences of the sequence are positive.
Next, the properties of the elements in the second block are analyzed.
It is shown that for sufficiently small $\lambda>0$, the sequence of elements for $n\in[k+1,2k]$ is strictly decreasing and also concave.
The purpose of the analysis is to help determine a supremum value for $\lambda$,
such that the pmf of the Poisson distribution of order $k\ge2$ decreases strictly for all $n \ge k$.
A conjectured criterion for the supremum is given.
Numerical calculations indicate it is the optimal bound, i.e.~the supremum.
In addition, a simple expression is proposed, based on numerical calculation, which is sufficient (but not necessary) and is a good approximation for the supremum.
\end{abstract}

\vskip 0.25in


\end{frontmatter}

\newpage
\setcounter{equation}{0}
\section{\label{sec:intro} Introduction}
The Poisson distribution of order $k$ is a special case of a compound Poisson distribution introduced by Adelson \cite{Adelson1966}.
The definition below is from \cite{KwonPhilippou}, with slight changes of notation.
\begin{definition}
  \label{def:pmf_Poisson_order_k}
  The Poisson distribution of order $k$ (where $k\ge1$ is an integer) and parameter $\lambda > 0$
  is an integer-valued statistical distribution with the probability mass function (pmf)
\bq
\label{eq:pmf_Poisson_order_k}
f_k(n;\lambda) = e^{-k\lambda}\sum_{n_1+2n_2+\dots+kn_k=n} \frac{\lambda^{n_1+\dots+n_k}}{n_1!\dots n_k!} \,, \qquad n=0,1,2\dots
\eq
\end{definition}
\noindent
For $k=1$ it is the standard Poisson distribution.
In a recent note \cite{Mane_Poisson_k_CC23_6}, the author presented a graphical analysis of the structure of the pmf of the Poisson distribution of order $k\ge2$.
For example, the pmf of the Poisson distribution of order $k\ge2$ can exhibit a maximum of {\em four} peaks simultaneously
(although not all are modes, i.e.~global maxima).
Also in a recent note \cite{Mane_Poisson_k_CC23_7},
the author presented analytical proofs of some of the properties of the structure of the pmf of the Poisson distribution of order $k$.

This note focuses on the properties of two blocks of elements of the probability mass function of the Poisson distribution of order $k\ge2$.
The first block is the elements for $n\in[1,k]$ and the second block is the elements for $n\in[k+1,2k]$.
The first goal of this note is to prove that elements in the first block form an ``absolutely monotonic sequence''
by which is meant that all the finite differences of the sequence are positive.
Next, we analyze the properties of the elements in the second block.
It will be shown that for sufficiently small $\lambda>0$, the sequence of elements for $n\in[k+1,2k]$ is strictly decreasing and also concave.
The purpose of the analysis is to help determine a supremum value for $\lambda$,
such that the pmf of the Poisson distribution of order $k\ge2$ decreases strictly for all $n \ge k$.
A conjectured criterion for the supremum is given in Conjecture \ref{conj:upbound_monotone_decrease} below.
Numerical calculations indicate it is the optimal bound, i.e.~the supremum.
We also propose a simple expression, based on numerical calculation, which is sufficient (but not necessary) and is a good approximation for the supremum.

The structure of this paper is as follows.
Sec.~\ref{sec:notation} presents basic definitions and notation employed in this note.
Sec.~\ref{sec:block1k} proves the absolute monotonicity of the probability mass function for $n\in[1,k]$.
Sec.~\ref{sec:block_kp1_2k} analyses the properties of the probability mass function for $n\in[k+1,2k]$.
A conjecture is proposed for the supremum value of $\lambda$, such that the pmf of the Poisson distribution of order $k$ decreases monotonically for all $n\ge k$.
Sec.~\ref{sec:numest_lam_pk1k2} proposes a simple expression, based on numerical calculation, which is sufficient (but not necessary) and is a good approximation for the supremum.
Sec.~\ref{sec:conc} concludes.

\newpage
\setcounter{equation}{0}
\section{\label{sec:notation}Basic notation and definitions}
\subsection{\label{sec:definitions}Definitions}
We work with $h_k(n;\lambda) = e^{k\lambda}f_k(n;\lambda)$ (see \cite{KwonPhilippou}) and refer to it as the ``scaled pmf'' below.
For later reference we define the parameter $\kappa=k(k+1)/2$.
From \cite{KwonPhilippou}, define $r_k$ as the value of $\lambda$ such that $h_k(k;\lambda)=1$.
Also for later reference, define $t_k$ as the value of $\lambda$ such that $h_k(k;\lambda)=2$ (see \cite{Mane_Poisson_k_CC23_7}).
Define the upper bound $\lambda_k^{\rm dec}$ such that for $\lambda\in(0,\lambda_k^{\rm dec})$,
the pmf of the Poisson distribution of order $k$ decreases strictly for all $n \ge k$.
It was shown in \cite{Mane_Poisson_k_CC23_7} that $\lambda_k^{\rm dec} > 0$.

For most of this note, we shall hold $k\ge2$ and $\lambda>0$ fixed and vary only the value of $n$.
For brevity of the exposition, we adopt the notation by Kostadinova and Minkova \cite{KostadinovaMinkova2013}
and write ``$p_n$'' in place of $h_k(n;\lambda)$ and omit explicit mention of $k$ and $\lambda$. 
Then $p_n$ is a polyomial in $\lambda$ of degree $n$ and for $n>0$ it has no constant term.
The recurrence for $p_n$ is as follows
(eq.~(6) in \cite{Adelson1966},
eq.~(6) in \cite{KwonPhilippou},
eq.~(8) and Remark 3 in \cite{KostadinovaMinkova2013},
eq.~(2.3) in \cite{GeorghiouPhilippouSaghafi}, 
and in all cases terms with negative indices are set to zero).
\bq
\label{eq:KM_rechk}
p_n = \frac{\lambda}{n} \,\sum_{j=1}^k jp_{n-j} \,.
\eq
Kostadinova and Minkova published the following expression for the pmf, in terms of combinatorial sums
(Theorem 1 in \cite{KostadinovaMinkova2013}, omitting the prefactor of $e^{-k\lambda}$).
\bq
\label{eq:KM_Thm1}
\begin{split}
  p_0 &= 1 \,,
  \\
  p_n &= \sum_{j=1}^n \binom{n-1}{j-1}\,\frac{\lambda^j}{j!} \qquad (n=1,2,\dots,k) \,,
  \\
  p_n &= \sum_{j=1}^n \binom{n-1}{j-1}\,\frac{\lambda^j}{j!}
  \;\;-\;\; \sum_{i=1}^\ell (-1)^{i-1}\,\frac{\lambda^i}{i!} \sum_{j=0}^{n-i(k+1)} \binom{n - i(k+1) +i-1}{j+i-1}\,\frac{\lambda^j}{j!}
  \\
  &\qquad\qquad (n = \ell(k+1)+m,\, m=0,1,\dots,k,\, \ell=1,2,\dots,\infty) \,.
\end{split}
\eq
Note the following.
\begin{enumerate}
\item
  The first element is just the single number $p_0=1$.
\item
  The block $n\in[1,k]$ has $k$ elements and is special in that there are no subtractions.
\item
  The tail piece for $n>k$ is composed of blocks of $k+1$ elements each, for $\ell=1,2,\dots$.
\item
  Numerical computations confirm eqs.~\eqref{eq:KM_rechk} and \eqref{eq:KM_Thm1} yield equal values for $p_n$.
\end{enumerate}

\newpage
\subsection{\label{sec:validation}Validation check of combinatorial sums for the probability mass function}
Let us validate Kostadinova and Minkova's expression for the scaled pmf in eq.~\eqref{eq:KM_Thm1} for $k=2$ and a few values of $n$.
For $k=2$ the expression for $p_n$ is simple, for all $n\ge1$.
\bq
\label{eq:pmf_k2}
\begin{split}
p_n &= \sum_{j=0}^{\lfloor(n/2)\rfloor} \frac{\lambda^{n-j}}{(n-2j)!j!}
\\
&= \frac{\lambda^n}{n!} +\frac{\lambda^{n-1}}{(n-2)!1!} +\frac{\lambda^{n-2}}{(n-4)!2!} +\dots +\frac{\lambda^{n-\lfloor(n/2)\rfloor}}{\lfloor(n/2)\rfloor!} \,.
\end{split}
\eq
The first few polynomials are as follows.
\begin{subequations}
\label{eq:poly_k2}
\begin{align}
  p_0 &= 1 \,,
  \\
  p_1 &= \lambda \,,
  \\
  p_2 &= \frac{\lambda^2}{2!} + \frac{\lambda}{0!1!} \,,
  \\
  p_3 &= \frac{\lambda^3}{3!} + \frac{\lambda^2}{1!1!} \,,
  \\
  p_4 &= \frac{\lambda^4}{4!} + \frac{\lambda^3}{2!1!} + \frac{\lambda^2}{0!2!} \,,
  \\
  p_5 &= \frac{\lambda^5}{5!} + \frac{\lambda^4}{3!1!} + \frac{\lambda^3}{1!2!} \,,
  \\
  p_6 &= \frac{\lambda^6}{6!} + \frac{\lambda^5}{4!1!} + \frac{\lambda^4}{2!2!} + \frac{\lambda^3}{0!3!} \,,
  \\
  p_7 &= \frac{\lambda^7}{7!} + \frac{\lambda^6}{5!1!} + \frac{\lambda^5}{3!2!} + \frac{\lambda^4}{1!3!} \,,
  \\
  p_8 &= \frac{\lambda^8}{8!} + \frac{\lambda^7}{6!1!} + \frac{\lambda^6}{4!2!} + \frac{\lambda^5}{2!3!} + \frac{\lambda^4}{0!4!} \,.
\end{align}
\end{subequations}
Let us employ eq.~\eqref{eq:KM_Thm1} and compare with the expressions in eq.~\eqref{eq:poly_k2}.
\begin{enumerate}
\item
  Case $n=1$:
\bq
\begin{split}
  p_1 &= \sum_{j=1}^1 \binom{1-1}{j-1}\,\frac{\lambda^j}{j!}
  = \lambda \,.
\end{split}
\eq
\item
  Case $n=2$:
\bq
\begin{split}
  p_2 &= \sum_{j=1}^2 \binom{2-1}{j-1}\,\frac{\lambda^j}{j!}
  = \binom{1}{0}\lambda +\binom{1}{1}\frac{\lambda^2}{2!}
  = \frac{\lambda}{0!1!} +\frac{\lambda^2}{2!0!} \,.
\end{split}
\eq
\item
  Case $n=3$:
\bq
\begin{split}
  p_3 &= \sum_{j=1}^3 \binom{2}{j-1}\,\frac{\lambda^j}{j!}
  \;-\; \sum_{i=1}^1 (-1)^{i-1}\,\frac{\lambda^i}{i!} \sum_{j=0}^{3-3i} \binom{3 - 3i +i-1}{j+i-1}\,\frac{\lambda^j}{j!}
  \\
  &= \binom{2}{0}\lambda +\binom{2}{1}\frac{\lambda^2}{2!} +\binom{2}{2}\frac{\lambda^3}{3!}  - \lambda
  \\
  &= \frac{\lambda^2}{1!1!} +\frac{\lambda^3}{3!0!} \,.
\end{split}
\eq

\newpage
\item
  Case $n=4$:
\bq
\begin{split}
  p_4 &= \sum_{j=1}^4 \binom{3}{j-1}\,\frac{\lambda^j}{j!}
  \;-\; \sum_{i=1}^1 (-1)^{i-1}\,\frac{\lambda^i}{i!} \sum_{j=0}^{4-3i} \binom{4 - 3i +i-1}{j+i-1}\,\frac{\lambda^j}{j!}
  \\
  &= \binom{3}{0}\lambda +\binom{3}{1}\frac{\lambda^2}{2!} +\binom{3}{2}\frac{\lambda^3}{3!} +\binom{3}{3}\frac{\lambda^4}{4!}
  -\lambda \sum_{j=0}^{1} \binom{1}{j}\,\frac{\lambda^j}{j!}
  \\
  &= \lambda + \frac{3\lambda^2}{0!2!} +\frac{\lambda^3}{2!1!} +\frac{\lambda^4}{4!0!} -\lambda (1 + \lambda)
  \\
  &= \frac{\lambda^2}{0!2!} +\frac{\lambda^3}{2!1!} +\frac{\lambda^4}{4!0!} \,.
\end{split}
\eq

\item
  Case $n=5$:
\bq
\begin{split}
  p_5 &= \sum_{j=1}^5 \binom{4}{j-1}\,\frac{\lambda^j}{j!}
  \;-\; \sum_{i=1}^1 (-1)^{i-1}\,\frac{\lambda^i}{i!} \sum_{j=0}^{5-3i} \binom{5 - 3i +i-1}{j+i-1}\,\frac{\lambda^j}{j!}
  \\
  &= \binom{4}{0}\lambda +\binom{4}{1}\frac{\lambda^2}{2!} +\binom{4}{2}\frac{\lambda^3}{3!} +\binom{4}{3}\frac{\lambda^4}{4!} +\binom{4}{4}\frac{\lambda^5}{5!}
  -\lambda \sum_{j=0}^{2} \binom{2}{j}\,\frac{\lambda^j}{j!}
  \\
  &= \lambda + \frac{4\lambda^2}{2!} +\frac{6\lambda^3}{3!} +\frac{4\lambda^4}{4!} +\frac{\lambda^5}{5!}
  -\lambda \Bigl(1 + 2\lambda +\frac{\lambda^2}{2!}\Bigr)
  \\
  &= \frac{\lambda^3}{1!2!} +\frac{\lambda^4}{3!1!} +\frac{\lambda^5}{5!0!} \,.
\end{split}
\eq

\item
  Case $n=6$:
\bq
\begin{split}
  p_6 &= \sum_{j=1}^6 \binom{5}{j-1}\,\frac{\lambda^j}{j!}
  \;-\; \sum_{i=1}^2 (-1)^{i-1}\,\frac{\lambda^i}{i!} \sum_{j=0}^{6-3i} \binom{6 - 3i +i-1}{j+i-1}\,\frac{\lambda^j}{j!}
  \\
  &= \binom{5}{0}\lambda +\binom{5}{1}\frac{\lambda^2}{2!} +\binom{5}{2}\frac{\lambda^3}{3!} +\binom{5}{3}\frac{\lambda^4}{4!}
  +\binom{5}{4}\frac{\lambda^5}{5!} +\binom{5}{5}\frac{\lambda^6}{6!}
  \\
  &\quad
  -\lambda \sum_{j=0}^{3} \binom{3}{j}\,\frac{\lambda^j}{j!}
  +\frac{\lambda^2}{2!} \sum_{j=0}^{0} \binom{1}{j+1}\,\frac{\lambda^j}{j!}
  \\
  &= \lambda + \frac{5\lambda^2}{2!} +\frac{10\lambda^3}{3!} +\frac{10\lambda^4}{4!} +\frac{5\lambda^5}{5!} +\frac{\lambda^6}{6!}
  \\
  &\quad
  -\lambda \Bigl(1 + 3\lambda +\frac{3\lambda^2}{2!} +\frac{\lambda^3}{3!}\Bigr)
  +\frac{\lambda^2}{2!} 
  \\
  &= \frac{\lambda^3}{0!3!} +\frac{\lambda^4}{2!2!} +\frac{\lambda^5}{4!1!} +\frac{\lambda^6}{6!0!} \,.
\end{split}
\eq

\newpage
\item
  Case $n=7$:
\bq
\begin{split}
  p_7 &= \sum_{j=1}^7 \binom{6}{j-1}\,\frac{\lambda^j}{j!}
  \;-\; \sum_{i=1}^2 (-1)^{i-1}\,\frac{\lambda^i}{i!} \sum_{j=0}^{7-3i} \binom{7 - 3i +i-1}{j+i-1}\,\frac{\lambda^j}{j!}
  \\
  &= \binom{6}{0}\lambda +\binom{6}{1}\frac{\lambda^2}{2!} +\binom{6}{2}\frac{\lambda^3}{3!} +\binom{6}{3}\frac{\lambda^4}{4!}
  +\binom{6}{4}\frac{\lambda^5}{5!} +\binom{6}{5}\frac{\lambda^6}{6!} +\binom{6}{6}\frac{\lambda^7}{7!}
  \\
  &\quad
  -\lambda \sum_{j=0}^{4} \binom{4}{j}\,\frac{\lambda^j}{j!}
  +\frac{\lambda^2}{2!} \sum_{j=0}^{1} \binom{2}{j+1}\,\frac{\lambda^j}{j!}
  \\
  &= \lambda + \frac{6\lambda^2}{2!} +\frac{15\lambda^3}{3!} +\frac{20\lambda^4}{4!} +\frac{15\lambda^5}{5!} +\frac{6\lambda^6}{6!} +\frac{\lambda^7}{7!}
  \\
  &\quad
  -\lambda \Bigl(1 + 4\lambda +\frac{6\lambda^2}{2!} +\frac{4\lambda^3}{3!} +\frac{\lambda^4}{4!}\Bigr)
  +\frac{\lambda^2}{2!} (2 +\lambda)
  \\
  &= \frac{\lambda^4}{1!3!} +\frac{\lambda^5}{3!2!} +\frac{\lambda^6}{5!1!} +\frac{\lambda^7}{7!0!} \,.
\end{split}
\eq

\item
  Case $n=8$:
\bq
\begin{split}
  p_8 &= \sum_{j=1}^8 \binom{7}{j-1}\,\frac{\lambda^j}{j!}
  \;-\; \sum_{i=1}^2 (-1)^{i-1}\,\frac{\lambda^i}{i!} \sum_{j=0}^{8-3i} \binom{8 - 3i +i-1}{j+i-1}\,\frac{\lambda^j}{j!}
  \\
  &= \binom{7}{0}\lambda +\binom{7}{1}\frac{\lambda^2}{2!} +\binom{7}{2}\frac{\lambda^3}{3!} +\binom{7}{3}\frac{\lambda^4}{4!}
  +\binom{7}{4}\frac{\lambda^5}{5!} +\binom{7}{5}\frac{\lambda^6}{6!} +\binom{7}{6}\frac{\lambda^7}{7!} +\binom{7}{7}\frac{\lambda^8}{8!}
  \\
  &\quad
  -\lambda \sum_{j=0}^{5} \binom{5}{j}\,\frac{\lambda^j}{j!}
  +\frac{\lambda^2}{2!} \sum_{j=0}^{2} \binom{3}{j+1}\,\frac{\lambda^j}{j!}
  \\
  &= \lambda + \frac{7\lambda^2}{2!} +\frac{21\lambda^3}{3!} +\frac{35\lambda^4}{4!} +\frac{35\lambda^5}{5!} +\frac{21\lambda^6}{6!} +\frac{7\lambda^7}{7!} +\frac{\lambda^8}{8!}
  \\
  &\quad
  -\lambda \Bigl(1 + 5\lambda +\frac{10\lambda^2}{2!} +\frac{10\lambda^3}{3!} +\frac{5\lambda^4}{4!} +\frac{\lambda^5}{5!}\Bigr)
  +\frac{\lambda^2}{2!} \Bigl(3 +3\lambda +\frac{\lambda^2}{2!}\Bigr)
  \\
  &= \frac{\lambda^4}{0!4!} +\frac{\lambda^5}{2!3!} +\frac{\lambda^6}{4!2!} +\frac{\lambda^7}{6!1!} +\frac{\lambda^8}{8!0!} \,.
\end{split}
\eq
\end{enumerate}

\newpage
\setcounter{equation}{0}
\section{\label{sec:block1k}Absolute monotonicity of the probability mass function for $n\in[1,k]$}
It was proved in Lemma 1 in \cite{KwonPhilippou} that for fixed $k\ge2$ and all $\lambda>0$,
the sequence $\{p_1,\dots,p_k\}$ is strictly increasing.
The following expression is from eq.~(3.3) in \cite{Mane_Poisson_k_CC23_7}.
\bq
\label{eq:KP_incseq_pn}  
\lambda = p_1 < p_2 < \dots < p_k \,.
\eq
That is to say, $p_n - p_{n-1} > 0$ for all $n\in[2,k]$.
We shall show that, in addition, {\em all} the finite differences are strictly positive for fixed $k\ge2$ and $\lambda>0$
and a suitable subinterval of $n\in[1,k]$ (see below).
The finite differences are analogs of derivatives, but for discrete (integer) values of $n$.
They are centered finite differences, for $m=1,2,\dots$ (with the formal definition $\Delta_0(n) = p_n$).
\bq
\label{eq:Delta_m_def}
\Delta_m(n) = \sum_{j=0}^m \binom{m}{j} (-1)^{j-1} p_{n-j} \,.
\eq
The first few examples are as follows.
\begin{subequations}
\label{eq:fd_1_2}
\begin{align}
  \Delta_1(n) &= p_n - p_{n-1} \,,
  \\
  \Delta_2(n) &= p_n - 2p_{n-1} + p_{n-2} \,.
\end{align}
\end{subequations}
In this section, we require only the expression for $p_n$ for $n\in[1,k]$ in eq.~\eqref{eq:KM_Thm1}.
Let us examine a few simple cases, to clarify ideas.
First recall the Pascal triangle recurrence for the binomial coefficients.
\bq
\label{eq:binom_recurrence}
\binom{n}{j} = \binom{n-1}{j} + \binom{n-1}{j-1} \,.
\eq  
For the first finite difference $\Delta_1(n)$ we obtain
\bq
\begin{split}
  p_n -p_{n-1} &= \biggl[\,\sum_{j=1}^n \binom{n-1}{j-1}\,\frac{\lambda^j}{j!}\,\biggr] - \biggl[\,\sum_{j=1}^{n-1} \binom{n-2}{j-1}\,\frac{\lambda^j}{j!}\,\biggr]
  \\
  &= \frac{\lambda^n}{n!} + \sum_{j=1}^{n-1} \biggl[\,\binom{n-1}{j-1} - \binom{n-2}{j-1}\,\biggr]\,\frac{\lambda^j}{j!}
  \\
  &= \frac{\lambda^n}{n!} + \sum_{j=2}^{n-1} \binom{n-2}{j-2}\,\frac{\lambda^j}{j!}
  \\
  &= \sum_{j=2}^n \binom{n-2}{j-2}\,\frac{\lambda^j}{j!} \,.
\end{split}
\eq
Hence $\Delta_1(n) > 0$.
The lower limit of the sum was increased from $j=1$ to $j=2$ because the binomial coefficients both equal $1$ for $j=1$ and cancel to zero.
We require $n\in[2,k]$ to justify the derivation.
For the second finite difference $\Delta_2(n)$ we obtain
\bq
\begin{split}
  p_n -2p_{n-1} +p_{n-2} &= (p_n -p_{n-1}) - (p_{n-1}-p_{n-2})
  \\
  &= \biggl[\,\sum_{j=2}^n \binom{n-2}{j-2}\,\frac{\lambda^j}{j!}\,\biggr] - \biggl[\,\sum_{j=2}^{n-1} \binom{n-3}{j-2}\,\frac{\lambda^j}{j!}\,\biggr]
  \\
  &= \frac{\lambda^n}{n!} + \sum_{j=2}^{n-1} \biggl[\,\binom{n-2}{j-2} - \binom{n-3}{j-2}\,\biggr]\,\frac{\lambda^j}{j!}
  \\
  &= \frac{\lambda^n}{n!} + \sum_{j=3}^{n-1} \binom{n-3}{j-3}\,\frac{\lambda^j}{j!}
  \\
  &= \sum_{j=3}^n \binom{n-3}{j-3}\,\frac{\lambda^j}{j!} \,.
\end{split}
\eq
Hence $\Delta_2(n) > 0$.
The lower limit of the sum was increased from $j=2$ to $j=3$ because the binomial coefficients both equal $1$ for $j=2$ and cancel to zero.
We require $n\in[3,k]$ to justify the derivation.
For the $m^{th}$ finite difference one can discern the pattern and surmise the answer.
\begin{proposition}
\label{prop:monotone_m}
For fixed $k\ge2$ and $\lambda>0$, the $m^{th}$ finite difference $\Delta_m(n)$ (see eq.~\eqref{eq:Delta_m_def})
is given as follows and is strictly positive for all $m\in[1,k-1]$ and all $n\in[m+1,k]$.
\bq
\label{eq:monotone_m}
\Delta_m(n) = \sum_{j=m+1}^n \binom{n-m-1}{j-m-1}\,\frac{\lambda^j}{j!} \;\; >\;\; 0\,.
\eq
\end{proposition}
\begin{proof}
We employ an induction argument on $m$.
Assume eq.~\eqref{eq:monotone_m} to be true for $m-1$ and $n\in[m,k]$ (where $k > m$ and $\lambda>0$ are fixed).
Then for $n\in[m+1,k]$ (which is a nonempty interval because $k>m$)
\bq
\begin{split}
  \sum_{j=0}^m \binom{m}{j} (-1)^{j-1} p_{n-j} &=
  \biggl[\,\sum_{j=m}^n \binom{n-m}{j-m}\,\frac{\lambda^j}{j!}\,\biggr]
  - \biggl[\,\sum_{j=m}^{n-1} \binom{n-m-1}{j-m}\,\frac{\lambda^j}{j!}\,\biggr]
  \\
  &= \frac{\lambda^n}{n!} + \sum_{j=m}^{n-1} \biggl[\,\binom{n-m}{j-m} - \binom{n-m-1}{j-m}\,\biggr]\,\frac{\lambda^j}{j!}
  \\
  &= \frac{\lambda^n}{n!} + \sum_{j=m+1}^{n-1} \binom{n-m-1}{j-m-1}\,\frac{\lambda^j}{j!}
  \\
  &= \sum_{j=m+1}^n \binom{n-m-1}{j-m-1}\,\frac{\lambda^j}{j!} \,.
\end{split}
\eq
We know the result to be true for $m=1$, hence the proof follows by induction on $m$.
\end{proof}
\newpage
\noindent
Feller \cite{Feller} defined a ``completely monotonic function'' $f(x)$ as one with the following property.
\begin{definition}
  A function $f:(0,\infty)\to[0,\infty)$ is completely monotonic if 
\bq
(-1)^nf^{(n)}(x) \ge 0 
\eq  
for all $n=0,1,2\dots$.
\end{definition}
\noindent
Note the following:
\begin{enumerate}
\item
  Feller's work was in connection with probability distributions.
\item
  I suspect Feller had in mind the negative exponential $e^{-x}$ is completely monotonic for $x\in(0,\infty)$.
\item
  A function such as $e^x$ has all positive derivatives but is technically not completely monotonic.
\end{enumerate}
Nevertheless, the underlying idea for us is that all the derivatives are positive (or zero).
We also replace derivatives with finite differences.
\begin{remark}
\label{rem:absolute_monotone}  
  We can perhaps say that the sequence $\{p_1,\dots,p_k\}$ is ``absolutely monotonic'' for all finite differences $m=1,\dots,k-1$,
  for fixed $k\ge2$ and $\lambda>0$.
\end{remark}
\noindent
Note that the term ``absolutely monotonic'' is {\em not} standard mathematical terminology.

\newpage
\setcounter{equation}{0}
\section{\label{sec:block_kp1_2k}Properties of the probability mass function for $n\in[k+1,2k]$}
Based on numerical calculations, it was noted in \cite{Mane_Poisson_k_CC23_6} that for fixed $k\ge2$ and sufficiently small $\lambda>0$,
the pmf of the Poisson distribution of order $k$ decreases monotonically for all $n \ge k$.
An analytical proof of this observation was published in \cite{Mane_Poisson_k_CC23_7}.
The proof in \cite{Mane_Poisson_k_CC23_7} employed induction on $n$.
\begin{proposition}  
\label{prop:inductionproof}(Restatement of Prop.~(4.1) in \cite{Mane_Poisson_k_CC23_7})
For fixed $k\ge2$, and $n \ge 2k$, suppose that there exists a fixed $\lambda>0$ such that the block of $k+1$ contiguous elements
$\{p_{n-k},p_{n-k+1},\dots,p_n\}$ form a strictly decreasing sequence $p_{n-k} > p_{n-k+1} > \dots > p_n$.
Then $p_{n+1}-p_n < 0$, i.e.~the sequence can be extended to include $p_n > p_{n+1}$.
\end{proposition}
\noindent
Note the following.
\begin{enumerate}
\item
  It was shown in \cite{Mane_Poisson_k_CC23_7} that there exists a value $\lambda>0$ such that the block of elements
  $\{p_k,\dots,p_{2k}\}$ is a strictly decreasing sequence of $k+1$ elements.
\item
It was also shown in \cite{KwonPhilippou} that $p_k > p_{k+1}$ for $0 < \lambda \le r_k$.
Recall $r_k$ is the value of $\lambda$ such that $p_k=1$. 
\item
  An improved upper bound on $\lambda$ (such that $p_k > p_{k+1}$) was derived in \cite{Mane_Poisson_k_CC23_7}.
  Recall $t_k$ is the value of $\lambda$ such that $p_k=2$. Then $p_k > p_{k+1}$ for $0 < \lambda \le t_k$.
  Note that $t_k$ is a sufficient but not necessary upper bound on the value of $\lambda$.
\end{enumerate}
Hence the conditions for $p_k > p_{k+1}$ have been analyzed in detail.
Our focus here is on the block of $k$ elements $\{p_{k+1},\dots,p_{2k}\}$.
As a simple starting exercise, consider $k=3$ (for $k\le2$ there are too few points).
For $k=3$, the polynomials in the block $n\in[k+1,2k]$ are
\begin{subequations}
\begin{align}
  p_4 = h_3(4;\lambda) &= \frac{\lambda^4}{4!} +\frac{\lambda^3}{2!} +\frac{3\lambda^2}{2} \,,
  \\
  p_5 = h_3(5;\lambda) &= \frac{\lambda^5}{5!} +\frac{\lambda^4}{3!} +\lambda^3 +\lambda^2 \,,
  \\
  p_6 = h_3(6;\lambda) &= \frac{\lambda^6}{6!} +\frac{\lambda^5}{4!} +\frac{5\lambda^4}{12} +\frac{7\lambda^3}{6} +\frac{\lambda^2}{2} \,.
\end{align}
\end{subequations}
The difference $p_5-p_4$ is
\bq
\begin{split}
  p_5-p_4 &= \frac{\lambda^5}{120} +\frac{\lambda^4}{8} +\frac{\lambda^3}{2} -\frac{\lambda^2}{2} 
  \\
  &= \frac{\lambda^2}{2}\biggl(\frac{\lambda^3}{60} +\frac{\lambda^2}{4} +\lambda -1\biggr) \,.
\end{split}
\eq
This is negative for $\lambda\in(0,\frac12)$.
The positive real root for $p_5-p_4=0$ is $\lambda\simeq 0.82187688$.
Next, the difference $p_6-p_5$ is
\bq  
\begin{split}
  p_6-p_5 &= \frac{\lambda^6}{6!} +\frac{\lambda^5}{30} +\frac{\lambda^4}{4} +\frac{\lambda^3}{6} -\frac{\lambda^2}{2}
  \\
  &= \frac{\lambda^2}{2}\biggl(\frac{\lambda^4}{360} +\frac{\lambda^3}{15} +\frac{\lambda^2}{2} +\frac{\lambda}{3} -1\biggr) \,.
\end{split}
\eq
This is negative for $\lambda\in(0,1)$.
The positive real root for $p_6-p_5=0$ is $\lambda \simeq 1.061200075$.  
In general, $p_{n+1} < p_n$ for all $n\in[k,2k-1]$ if the first pair is decreasing: $p_{k+2} < p_{k+1}$.
We shall prove this below, for all $k\ge2$.
Next consider the convexity (second finite difference $\Delta_2(n)$), for $k=3$ and $n=6$:
(We require $k\ge3$ because for $k=2$ there are too few points in the interval $n\in[k+1,2k]$.)
\bq  
\begin{split}
  p_6+p_4 -2p_5 &= \frac{\lambda^6}{720} +\frac{\lambda^5}{40} +\frac{7\lambda^4}{24} -\frac{\lambda^3}{3} 
  \\
  &= \lambda^3\biggl(\frac{\lambda^3}{720} +\frac{\lambda^2}{40} +\frac{\lambda}{8} -\frac13\biggr) \,.
\end{split}
\eq
This is negative for small $\lambda>0$ (concave) but eventually positive (convex).
The positive real root is $\lambda \simeq 1.88318444$.  
Hence, unlike the sequence $\{p_1,\dots,p_k\}$, the sequence $\{p_{k+1},\dots,p_{2k}\}$ is not always strictly increasing or decreasing, nor always convex or concave.
Fig.~\ref{fig:graph_hist_k10_nonmonotonic} displays a plot of the scaled pmf $h_k(n;\lambda)$
of the Poisson distribution of order $10$ and $\lambda=0.3$ (circles) and $\lambda=0.2$ (triangles).
\begin{enumerate}
\item
For both $\lambda=0.2$ and $0.3$, the scaled pmf increases strictly for $n\in[1,k]$ (see Sec.~\ref{sec:block1k}).
\item
For $\lambda=0.2$, the scaled pmf decreases strictly for $n\in[k+1,2k]$ but for $\lambda=0.3$ it exhibits a local maximum (which is {\em not} a mode).
\item
Also for both $\lambda=0.2$ and $0.3$, the scaled pmf is concave for $n\in[k+1,2k]$.
\item
  For both $\lambda=0.2$ and $0.3$, notice the large drop in value from $p_k$ to $p_{k+1}$.
\item
  There is also a noticeable drop in value from $p_{2k}$ to $p_{2k+1}$, although (much?) smaller in magnitude.
\item
  The smallest power of $\lambda$ in $p_n$ for $n\in[(i-1)k+1,ik]$ is $\lambda^i$, for $i=1,2,\dots$.
  Hence the smallest power of $\lambda$ changes from $\lambda$ in $p_k$ to $\lambda^2$ in $p_{k+1}$ and from $\lambda^2$ in $p_{2k}$ to $\lambda^3$ in $p_{2k+1}$.
  There are probably similar drops in value between $p_{ik}$ and $p_{ik+1}$ for all $i=1,2,\dots$.
  However, this cannot be a general rule because for sufficiently large values of $\lambda$, the pmf does not decrease monotonically for $n\ge k$.
\end{enumerate}

\newpage
In the rest of this section, we quantify the behavior of the sequence $\{p_{k+1},\dots,p_{2k}\}$ in more detail.
We employ the expression for $p_n$ by Kostadinova and Minkova in eq.~\eqref{eq:KM_Thm1}.
For $n\in[k+1,2k]$ the expression is
\bq
\label{eq:KM_block_kp1_2k}
\begin{split}
  p_n = \sum_{j=1}^n \binom{n-1}{j-1}\,\frac{\lambda^j}{j!}
  \;-\; \lambda \sum_{j=0}^{n-k-1} \binom{n -k-1}{j}\,\frac{\lambda^j}{j!} \,.
\end{split}
\eq
Actually eq.~\eqref{eq:KM_block_kp1_2k} is valid also for $n=2k+1$, but we treat only the block $n\in[k+1,2k]$ here.
The lowest power of $\lambda$ in the block $n\in[k+1,2k]$ is $\lambda^2$.
The first finite difference $\Delta_1(n)$ for $n\in[k+2,2k]$ is
\bq
\begin{split}
  p_n - p_{n-1} &= \sum_{j=1}^n \binom{n-1}{j-1}\,\frac{\lambda^j}{j!} \;-\; \lambda \sum_{j=0}^{n-k-1} \binom{n-k-1}{j}\,\frac{\lambda^j}{j!}
  \\
  &\qquad -\sum_{j=1}^{n-1} \binom{n-2}{j-1}\,\frac{\lambda^j}{j!} \;+\; \lambda \sum_{j=0}^{n-k-2} \binom{n-k-2}{j}\,\frac{\lambda^j}{j!}
  \\
  &= \frac{\lambda^n}{n!} + \sum_{j=1}^{n-1} \biggl[\,\binom{n-1}{j-1} - \binom{n-2}{j-1}\,\biggr]\,\frac{\lambda^j}{j!} 
  \\
  &\qquad -\frac{\lambda^{n-k}}{(n-k-1)!} -\lambda\,\sum_{j=0}^{n-k-2} \biggl[\,\binom{n-k-1}{j} - \binom{n-k-2}{j}\,\biggr]\,\frac{\lambda^j}{j!}
  \\
  &= \frac{\lambda^n}{n!} + \sum_{j=2}^{n-1} \binom{n-2}{j-2}\,\frac{\lambda^j}{j!} 
  -\frac{\lambda^{n-k}}{(n-k-1)!} -\lambda\,\sum_{j=1}^{n-k-2} \binom{n-k-2}{j-1}\,\frac{\lambda^j}{j!}
  \\
  &= \sum_{j=2}^n \binom{n-2}{j-2}\,\frac{\lambda^j}{j!} -\lambda\,\sum_{j=1}^{n-k-1} \binom{n-k-2}{j-1}\,\frac{\lambda^j}{j!} \,.
\end{split}
\eq
Writing out the lowest powers of $\lambda$ yields
\bq
\begin{split}
  p_n -p_{n-1} &= \frac{\lambda^2}{2!} -\frac{\lambda^2}{1!} +(n-2)\frac{\lambda^3}{3!} -(n-k-2)\frac{\lambda^3}{2!} +\dots
  \\
  &= -\frac{\lambda^2}{2} +(n-2-3n+3k+6)\frac{\lambda^4}{4!} +\dots
  \\
  &= -\frac{\lambda^2}{2} +(3k+4-2n)\frac{\lambda^3}{3!} +\dots
\end{split}
\eq
For fixed $k\ge2$ and sufficiently small $\lambda>0$, the sequence $\{p_{k+1},\dots,p_{2k}\}$ is {\em strictly decreasing} in $n$.
(This is a simpler and better proof of the decreasing sequence property than that in \cite{Mane_Poisson_k_CC23_7}.)
The second difference $\Delta_2(n)$ is given by a similar induction argument to that for the block $n\in[1,k]$.
\bq
\begin{split}
  p_n -2p_{n-1} +p_{n-2} &= \sum_{j=3}^n \binom{n-3}{j-3}\,\frac{\lambda^j}{j!} -\lambda\,\sum_{j=2}^{n-k-1} \binom{n-k-3}{j-2}\,\frac{\lambda^j}{j!} \,.
\end{split}
\eq
This is valid for all $n\in[k+3,2k]$.
Writing out the lowest powers of $\lambda$ yields
\bq
\begin{split}
  p_n -2p_{n-1} +p_{n-2} &= \frac{\lambda^3}{3!} -\frac{\lambda^3}{2!} +(n-3)\frac{\lambda^4}{4!} -(n-k-3)\frac{\lambda^4}{3!} +\dots
  \\
  &= -\frac{\lambda^3}{3} +(n-3-4n+4k+12)\frac{\lambda^4}{4!} +\dots
  \\
  &= -\frac{\lambda^3}{3} +(4k+9-3n)\frac{\lambda^4}{4!} +\dots
\end{split}
\eq
For fixed $k\ge2$ and sufficiently small $\lambda>0$, the sequence $\{p_{k+1},\dots,p_{2k}\}$ is {\em concave} in $n$.
Hence for fixed $k\ge2$ and sufficiently small $\lambda>0$, if $p_{k+2} < p_{k+1}$, then $p_{n+1} < p_n$ for all $n\in[k+1,2k]$.
\begin{enumerate}
\item
  The concavity property suggests that if $p_{k+2} < p_{k+1}$, then $p_{n+1} < p_n$ for all $n\in[k+1,2k]$.
\item
  There is the additional condition $p_{k+1} < p_k$, then the pmf decreases strictly for all $n \ge k$ (the proof was given in \cite{Mane_Poisson_k_CC23_7}).
  It was shown in \cite{Mane_Poisson_k_CC23_7} that $p_{k+1} < p_k$ if $\lambda \le t_k$, but this is a sufficient but not necessary upper bound on $\lambda$.  
\end{enumerate}
It is still necessary to determine a quantitative upper bound for $\lambda$ such that $p_{n+1} < p_n$ for all $n\in[k+1,2k]$.
The evidence suggests that the condition $p_{k+1}=p_{k+2}$ yields the supremum for $\lambda$.
\begin{definition}
\label{def:def_lam_pk1k2}
Let $\lambda_{k+1,k+2}$ denote the (unique) positive real root of the equation $p_{k+2}-p_{k+1}=0$.
\end{definition}
\begin{conjecture}
\label{conj:upbound_monotone_decrease}
  For fixed $k\ge2$, the value of $\lambda_k^{\rm dec}$, the supremum for $\lambda$ such that the pmf of the Poisson distribution of order $k$
  decreases strictly for all $n\ge k$ is given by the positive real root of the equation $p_{k+2}-p_{k+1}=0$, denoted by $\lambda_{k+1,k+2}$:
\bq
\label{eq:conj_lambda_k_dec}
\lambda_k^{\rm dec} = \lambda_{k+1,k+2} \,.
\eq
Then for fixed $k\ge2$ and $\lambda\in(0,\lambda_{k+1,k+2})$, $p_{n+1} < p_n$ for all $n\ge k$.
\end{conjecture}
\begin{remark}
It was shown in \cite{Mane_Poisson_k_CC23_7} that $p_{k+1} < p_k$ if $\lambda \le t_k$, but this is a sufficient but not necessary bound.
Recall $t_k$ is the value of $\lambda$ such that $p_k=2$.
Numerical evidence suggests that $\lambda < \lambda_{k+1,k+2}$ is the applicable supremum to employ in Conjecture \ref{conj:upbound_monotone_decrease}.
For example for $k=2$, $t_k$ is the positive real root of the equation $\frac12\lambda^2 +\lambda - 2 = 0$ and its value is $t_k=\sqrt{5}-1 \simeq 1.23606$.
Also $\lambda_{k+1,k+2}$ is the positive real root of the equation $\frac16\lambda^3 +\frac12\lambda^2 -\lambda = 0$
and its value is $\lambda_{k+1,k+2}=2(\sqrt{7}-2) \simeq 1.2915026$.
Fig.~\ref{fig:graph_hist_k2_pk1k2} displays a plot of the scaled pmf $h_k(n;\lambda)$
of the Poisson distribution of order $2$ and $\lambda=\lambda_{k+1,k+2}=2(\sqrt{7}-2) \simeq 1.2915026$.
Observe that $p_k>2$ but even so, $p_{k+1} < p_k$.
Observe also that $p_{k+1}=p_{k+2}$ and the pmf decreases monotonically for $n \ge k$, except for the one case $p_{k+1}=p_{k+2}$.
\end{remark}
\begin{remark}
\label{rem:upbound_k3}
For $k=3$, the condition $p_{k+1}=p_{k+2}$ is attained for $\lambda = -5 +\sqrt[3]{55-10\sqrt{29}} +\sqrt[3]{55+10\sqrt{29}} \simeq 0.82187688$.
Fig.~\ref{fig:graph_hist_k3_pk1k2} displays a plot of the scaled pmf $h_k(n;\lambda)$
of the Poisson distribution of order $3$ and $\lambda=0.82187688$.
Observe that $p_{k+1}=p_{k+2}$ and the pmf decreases monotonically for $n \ge k$, except for the one case $p_{k+1}=p_{k+2}$.  
Observe also that $p_{k+1} < p_k$ but $p_k<2$.
This goes to show that the condition $p_k<2$, i.e.~$\lambda \le t_k$, is not relevant to determine a supremum value for $\lambda$
for the pmf to decrease monotonically for all $n \ge k$.
\end{remark}
\begin{remark}
Numerical calculations indicate eq.~\eqref{eq:conj_lambda_k_dec} yields the supremum for $\lambda$ for the pmf of the Poisson distribution of order $k$
to decrease strictly for all $n\ge k$.  
It remains to determine a precise expression for the value of $\lambda_{k+1,k+2}$.
By definition, it is the positive real root of a polynomial equation, but that is a polynomial equation of degree $k$.
\end{remark}
\noindent
We can obtain a gross {\em overestimate} for $\lambda_k^{\rm dec}$ as follows.
For the pmf to be strictly decreasing for all $n\ge k$, we must have $p_n<p_{n-1}$ for all $n \ge k+1$.
Process eq.~\eqref{eq:KM_rechk} using $p_n < p_j$ for $j < n$:
\bq
\label{eq:gross_upbound_monotone}
\begin{split}
  np_n &= \lambda\,(p_{n-1} +2p_{n-2} +\dots +kp_{n-k})
  \\
  &> \lambda p_n \,(1+\dots +k)
  \\
  &=\lambda\frac{k(k+1)}{2}\, p_n
  \\
  &=\kappa\lambda\, p_n \,.
\end{split}  
\eq
Cancel $p_n$ to deduce $\kappa\lambda < n$. We require $n-k\ge k$ to justify eq.~\eqref{eq:gross_upbound_monotone},
so the smallest applicable value of $n$ is $2k$.
Hence
\bq
  \lambda < \frac{2k}{\kappa} = \frac{4}{k+1} \,.
\eq
For $k=3$, the right-hand size equals $1$.
We saw in Remark \ref{rem:upbound_k3} above that the correct upper bound is lower than this.
See Fig.~\ref{fig:graph_hist_k3_pk1k2}.
\begin{proposition}
  For fixed $k\ge2$, an upper bound for $\lambda_k^{\rm dec}$, for the pmf of the Poisson distribution of order $k$ to decrease strictly for all $n\ge k$, is given by
\bq
\label{eq:lambda_dec_upbound}  
  \lambda_k^{\rm dec} < \frac{4}{k+1} \,.
\eq
This is a necessary but not sufficient upper bound.  
This upper bound shows that the value of $\lambda_k^{\rm dec}$ must decrease at least as $O(1/k)$ as $k$ increases.
\end{proposition}

\newpage
\setcounter{equation}{0}
\section{\label{sec:numest_lam_pk1k2}Numerical estimate of supremum for monotonic decrease of the pmf for $n\ge k$}
The value of $\lambda_{k+1,k+2}$ was computed numerically for $2 \le k \le 40000$.
The inverse value $\lambda_{k+1,k+2}^{-1}$ is fitted remarkably well by a straight line.
Fig.~\ref{fig:graph_invroot_pk1k2} displays a plot of $\lambda_{k+1,k+2}^{-1}$ (dotted line) for $2 \le k \le 40000$.
The dashed line is the straight line fit $\alpha k+\beta = 0.442972564\,k -0.113086$ and is visually indistinguishable.
The coefficients of the straight line were obtained via a series of regression fits and are given as follows
\begin{subequations}
\label{eq:fit_param_lam_pk1k2}
\begin{align}
  \alpha &= \frac49 - 0.0015 +3\cdot10^{-5} -2\cdot10^{-6} +1\cdot10^{-7} +2\cdot10^{-8} -9\cdot10^{-10} &=& \phantom{-}0.442972564 \,,
  \\
  \beta &= -0.1131 +2\cdot10^{-5} -6\cdot10^{-6} &=& -0.113086 \,.
\end{align}  
\end{subequations}
Fig.~\ref{fig:graph_invroot_diff_pk1k2} displays a plot of the difference $\alpha k +\beta - \lambda_{k+1,k+2}^{-1}$ for $1000 \le k \le 40000$.
Note the following.
\begin{enumerate}
\item
  The difference increases as the value of $k$ increases, which is to be expected from a numerical calculation.
\item
The maximum difference is $0.000282326$ and the minimum difference is $-0.000278234$, at the right-hand edge $k=40000$.
From Fig.~\ref{fig:graph_invroot_pk1k2}, $\lambda_{k+1,k+2}^{-1} \simeq 17718$ for $k=40000$, which yields a relative accuracy of $1.58\cdot10^{-8}$.
\item
The worst relative accuracy of the fit is $0.003$ at $k=3$.
\item
  Significantly, the difference is {\em symmetric} around zero, i.e.~the fit is {\em unbiased}: it is neither systematically too high nor too low.
\end{enumerate}
  One can conjecture from the values of $\alpha$ and $\beta$ in eq.~\eqref{eq:fit_param_lam_pk1k2}
  that the exact values are really $\alpha=\frac49$ and $\beta=-\frac19$, and the residuals are due to numerical precision errors.
Hence asymptotically $\lambda_{k+1,k+2} \simeq 9/(4k-1)$.
The numerical evidence suggests this is not quite correct, although it is a good approximation.
Table \ref{tb:lam_pk1k2} tabulates the values of $\lambda_{k+1,k+2}$ and $9/(4k-1)$ and the difference $\lambda_{k+1,k+2} - 9/(4k-1)$
for $k=2,\dots,10$ and higher values.  
Observe that $\lambda_{k+1,k+2} > 9/(4k-1)$ in all cases and the difference decreases as the value of $k$ increases.
Numerical calculations indicate that $\lambda_{k+1,k+2} > 9/(4k-1)$ for all $k\in[2,40000]$.
  Numerical calculations moreover indicate that for all $k\in[2,40000]$, the pmf decreases monotonically for $0 < \lambda < 9/(4k-1)$.
\begin{conjecture}
\label{conj:suff_not_nec_upbound_monotone_decrease}
For fixed $k\ge2$, the pmf of the Poisson distribution of order $k$ decreases monotonically for all $n\ge k$ for
\bq
\label{eq:suff_not_nec_upbound_monotone_decrease}
0 < \lambda < \frac{9}{4k-1} \,.
\eq
The value $9/(4k-1)$ is a sufficient, but not necessary, upper bound on the value of $\lambda$.
It is a good approximation to the true supremum $\lambda_k^{\rm dec}$ for $k\gg1$.
\end{conjecture}
\begin{remark}
  Increasing the value of $\lambda$ to $9.05/(4k-1)$ results in a value which is too large.
  For fixed $k\ge2$ and $\lambda=9.05/(4k-1)$, the pmf of the Poisson distribution of order $k$ is {\em not} decreasing for all $n\ge k$
  (for all tested values $k\ge2$).
  This suggests that $9/(4k-1)$ is a good approximation (sufficient but not necessary) for the true supremum $\lambda_k^{\rm dec}$.
\end{remark}

\newpage
\section{\label{sec:conc}Conclusion}
This note focused on the properties of two blocks of elements of the probability mass function (pmf) of the Poisson distribution of order $k\ge2$.
The first block was the elements $p_n$ for $n\in[1,k]$ and the second block was the elements $p_n$ for $n\in[k+1,2k]$.
The first major goal of this note was to prove that the elements $p_n$ for $n\in[1,k]$ form an ``absolutely monotonic sequence''
by which is meant that all the finite differences of the sequence are positive.
The second major goal was to analyze the properties of the elements $p_n$ for $n\in[k+1,2k]$.
It was shown that for sufficiently small $\lambda>0$, the sequence is strictly decreasing and also concave.
The purpose of the analysis was to help determine a supremum value for $\lambda$,
such that the pmf of the Poisson distribution of order $k\ge2$ decreases strictly for all $n \ge k$.
A conjectured criterion for the supremum was given in Conjecture \ref{conj:upbound_monotone_decrease}.
Numerical calculations indicate it is the optimal bound, i.e.~the supremum.
In addition, a simple expression was proposed, based on numerical calculation,
which is sufficient (but not necessary) and is a good approximation for the supremum.


\newpage
\begin{table}[htb]
\centering
\begin{tabular}[width=0.75\textwidth]{|r|l|l|l|}
  \hline
  $k$ & $\lambda_{k+1,k+2}$ & $9/(4k-1)$ & $\lambda_{k+1,k+2} - 9/(4k-1)$ \\  \hline
2	&	1.291502622	&	1.285714286	&	0.005788336	\\ \hline
3	&	0.821876885	&	0.818181818	&	0.003695066	\\ \hline
4	&	0.602607787	&	0.6	&	0.002607787	\\ \hline
5	&	0.475672588	&	0.473684211	&	0.001988378	\\ \hline
6	&	0.392901337	&	0.391304348	&	0.001596989	\\ \hline
7	&	0.334663355	&	0.333333333	&	0.001330022	\\ \hline
8	&	0.29145995	&	0.290322581	&	0.001137369	\\ \hline
9	&	0.258135147	&	0.257142857	&	0.00099229	\\ \hline
10	&	0.231648581	&	0.230769231	&	0.00087935	\\ \hline
100	&	0.022632529	&	0.022556391	&	$7.61377\cdot10^{-5}$	\\ \hline
1000	&	0.002258053	&	0.002250563	&	$7.48997\cdot10^{-6}$	\\ \hline
10000	&	0.000225753	&	0.000225006	&	$7.47754\cdot10^{-7}$	\\ \hline
20000	&	0.000112875	&	0.000112501	&	$3.73843\cdot10^{-7}$	\\ \hline
30000	&	$7.52498\cdot10^{-5}$	&	$7.50006\cdot10^{-5}$	&	$2.49221\cdot10^{-7}$	\\ \hline
40000	&	$5.64373\cdot10^{-5}$	&	$5.62504\cdot10^{-5}$	&	$1.86913\cdot10^{-7}$	\\ \hline
\end{tabular}
\caption{\label{tb:lam_pk1k2}
  Tabulation of $\lambda_{k+1,k+2}$ and the approximation $9/(4k-1)$ and the difference, for $k=2,\dots,10$ and selected higher values.}
\end{table}

\newpage
\begin{figure}[!htb]
\centering
\includegraphics[width=0.75\textwidth]{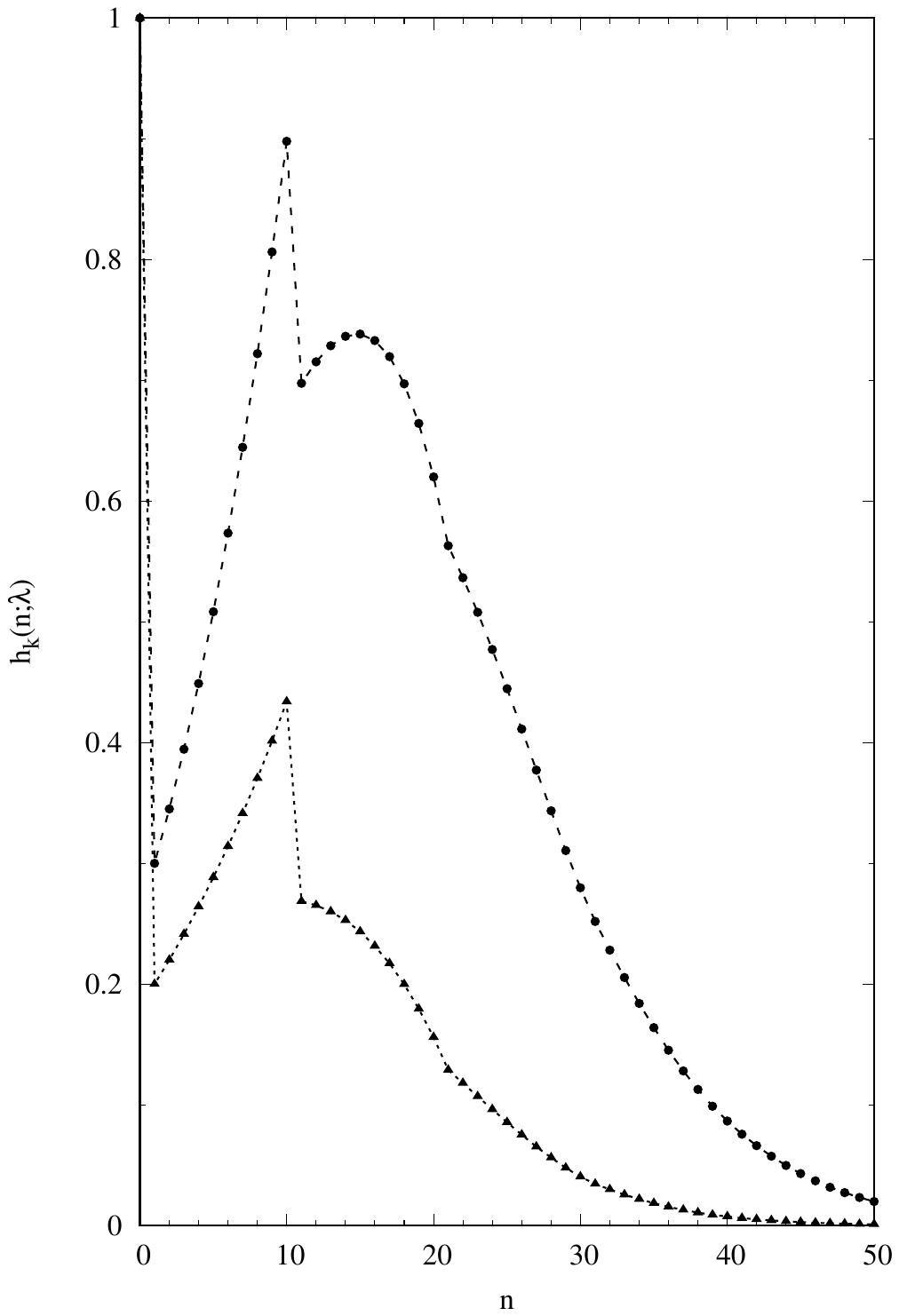}
\caption{\small
\label{fig:graph_hist_k10_nonmonotonic}
Plot of the scaled pmf $h_k(n;\lambda)$ of the Poisson distribution of order $10$ and $\lambda=0.3$ (circles) and $\lambda=0.2$ (triangles).}
\end{figure}

\newpage
\begin{figure}[!htb]
\centering
\includegraphics[width=0.75\textwidth]{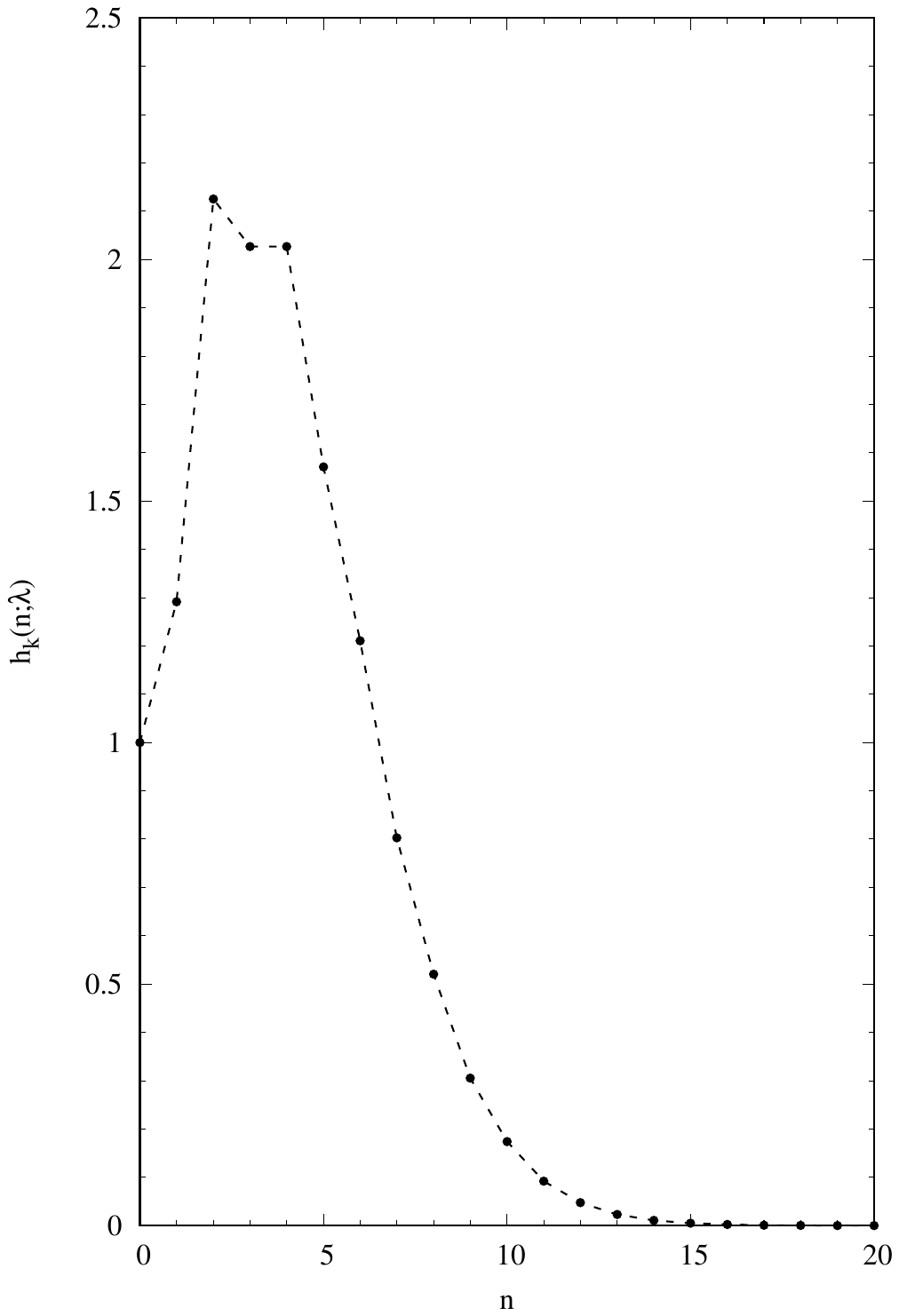}
\caption{\small
\label{fig:graph_hist_k2_pk1k2}
Plot of the scaled pmf $h_k(n;\lambda)$ of the Poisson distribution of order $2$ and $\lambda=1.2915026$.}
\end{figure}

\newpage
\begin{figure}[!htb]
\centering
\includegraphics[width=0.75\textwidth]{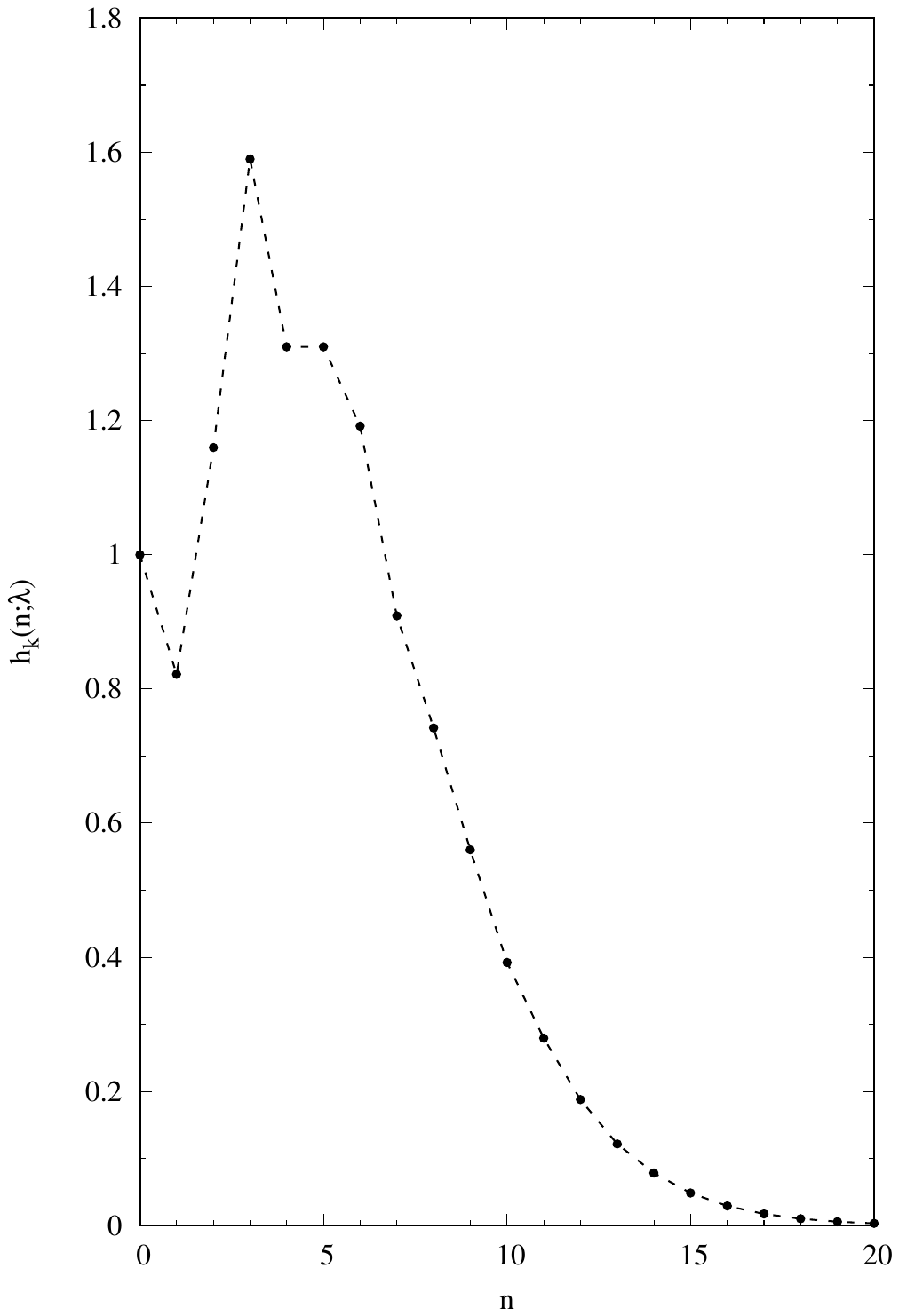}
\caption{\small
\label{fig:graph_hist_k3_pk1k2}
Plot of the scaled pmf $h_k(n;\lambda)$ of the Poisson distribution of order $3$ and $\lambda=0.82187688$.}
\end{figure}

\newpage
\begin{figure}[!htb]
\centering
\includegraphics[width=0.75\textwidth]{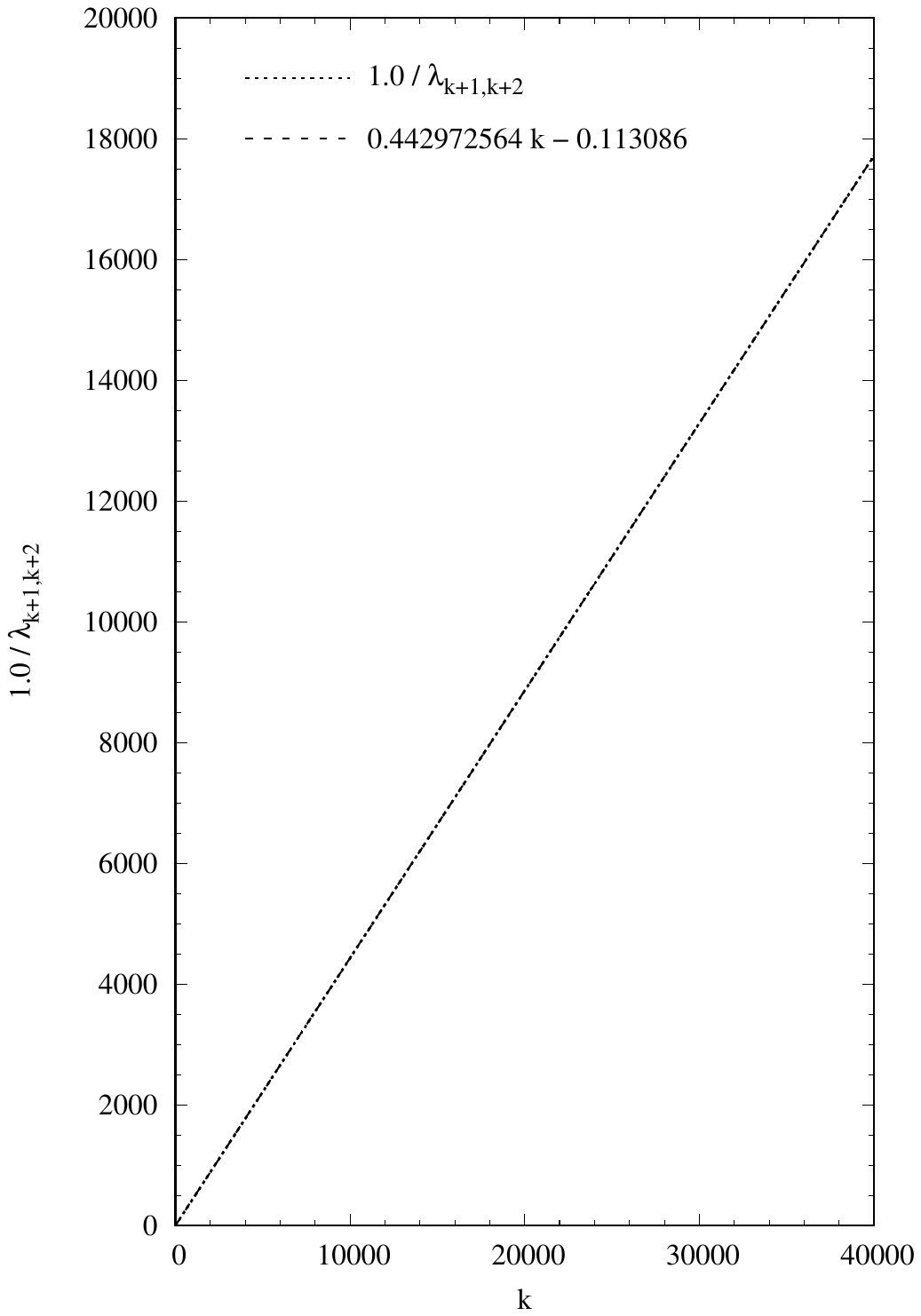}
\caption{\small
\label{fig:graph_invroot_pk1k2}
Plot of the inverse root $\lambda_{k+1,k+2}^{-1}$ (dotted line) and the fit function $0.442972564\,k -0.113086$ (dashed).}
\end{figure}

\newpage
\begin{figure}[!htb]
\centering
\includegraphics[width=0.75\textwidth]{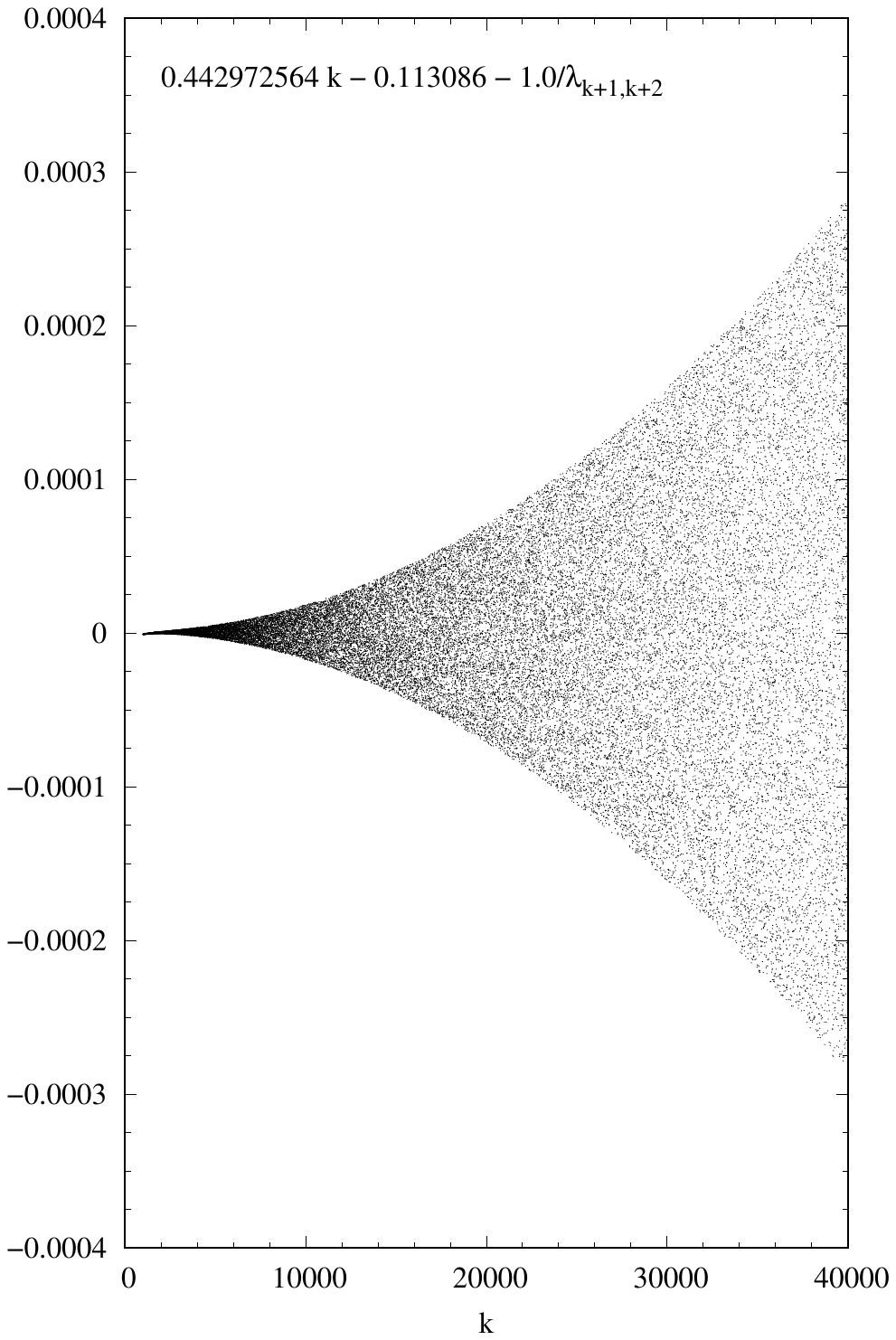}
\caption{\small
\label{fig:graph_invroot_diff_pk1k2}
Plot of the difference $0.442972564\,k -0.113086 - \lambda_{k+1,k+2}^{-1}$ for $1000 \le k \le 40000$.}
\end{figure}

\end{document}